\documentclass[12pt,leqno]{amsart}

\usepackage[T1]{fontenc}
\usepackage{lmodern}
\usepackage{amsmath,amssymb,mathtools}
\usepackage{microtype}
\usepackage[hidelinks]{hyperref}
\hypersetup{
  pdftitle={Stability in recovering the order of a time-fractional evolution equation},
  pdfauthor={Ravshan Ashurov and Masahiro Yamamoto}
}

\allowdisplaybreaks
\numberwithin{equation}{section}

\newtheorem{theorem}{Theorem}[section]
\newtheorem{proposition}[theorem]{Proposition}
\newtheorem{lemma}[theorem]{Lemma}
\newtheorem{corollary}[theorem]{Corollary}
\theoremstyle{definition}

\theoremstyle{remark}
\newtheorem{remark}[theorem]{Remark}

\newcommand{\cA}{\mathcal{A}}
\newcommand{\cP}{\mathcal{P}}
\newcommand{\cM}{\mathcal{M}}
\newcommand{\D}{\mathcal{D}}
\newcommand{\norm}[1]{\lVert #1\rVert}
\newcommand{\abs}[1]{\lvert #1\rvert}
\newcommand{\ip}[2]{\bigl(#1,#2\bigr)_X}
\newcommand{\dd}{\,\mathrm{d}}

\title[Stability in recovering a fractional order]
{Stability in recovering\\
the order of a time-fractional evolution equation}

\author[R. Ashurov]{Ravshan Ashurov}
\address{V. I. Romanovskiy Institute of Mathematics, Uzbekistan Academy of Sciences, 100174 Tashkent, Uzbekistan}
\address{School of Engineering, Central Asian University, 111221 Tashkent, Uzbekistan}
\author[M. Yamamoto]{Masahiro Yamamoto}
\address{Graduate School of Mathematical Sciences, The University of Tokyo,
Komaba, Meguro, Tokyo 153-8914, Japan}
\address{Department of Mathematics, Faculty of Arts and Sciences,
Zonguldak B\"ulent Ecevit University, Zonguldak 67100, Turkey}
\email{myama@ms.u-tokyo.ac.jp}

\subjclass[2020]{Primary 35R30; Secondary 26A33, 35R11, 47D06}
\keywords{time-fractional evolution equation, inverse problem, fractional order,
 stability, Mittag--Leffler function}

\begin{document}

\begin{abstract}
We consider the recovery of the fractional order in a time-fractional evolution
 equation with a source of the form $p(t)f$, where the temporal factor $p$ is
also unknown but belongs to a bounded admissible class and has a prescribed
nonzero value at $t=0$.  From one scalar observation we prove a uniform
small-time expansion and obtain a H\"older estimate
for the fractional order.  More precisely, if a fixed derivative order
$\gamma$ lies below every admissible fractional order, then the stability
exponent is $\alpha_+/(2\alpha_+-\gamma)$; for the undifferentiated observation
this gives the exponent $1/2$.  We also show that differentiating the data to
an order larger than all admissible fractional orders is supercritical: the
corresponding derivative is unbounded near $t=0$ whenever the two fractional
orders are different.  A uniform estimate
for the two-parameter Mittag--Leffler function, including the diagonal case,
is proved in the appendix.
\end{abstract}

\maketitle

\section{Introduction}

Let $X$ be a separable Hilbert space, and let $0<\alpha<1$. For
$w\in L^2(0,T;X)$, the fractional Riemann--Liouville integral of order
$\alpha$ is defined by
\[
J^\alpha w(t)
:=
\frac{1}{\Gamma(\alpha)}
\int_0^t (t-s)^{\alpha-1}w(s)\,ds,
\qquad 0<t<T,
\]
where the integral is understood in the sense of the Bochner integral;
see, e.g., \cite{Liz}. The operator
\[
J^\alpha:L^2(0,T;X)\longrightarrow L^2(0,T;X)
\]
is injective; see, e.g., \cite{7}. We therefore define the fractional
derivative of order $\alpha$ by
\[
\partial_t^\alpha:=(J^\alpha)^{-1},
\qquad
\mathcal{D}(\partial_t^\alpha)
=
J^\alpha L^2(0,T;X).
\]
Accordingly, we introduce the Banach space
\[
H_\alpha(0,T;X)
:=
J^\alpha L^2(0,T;X)
\]
equipped with the norm
\[
\|v\|_{H_\alpha(0,T;X)}
:=
\|(J^\alpha)^{-1}v\|_{L^2(0,T;X)}.
\]
If $\frac12<\alpha<1$, then functions belonging to
$H_\alpha(0,T;X)$ are continuous with respect to $t$ and vanish at
$t=0$. In contrast, for $0<\alpha<\frac12$, functions in
$H_\alpha(0,T;X)$ are generally not continuous, and consequently their
value at $t=0$ cannot, in general, be defined; see, for example,
\cite{KubicaRyszewskaYamamoto,Ya21,7}.

Let $A$ be a self-adjoint negative operator in $X$ with domain
$\mathcal{D}(A)$. We consider the initial value problem
\begin{equation}\label{eq:forward}
\partial_t^\alpha u(t)
=
Au(t)+p(t)f,
\qquad
t\in(0,T),
\qquad
u\in H_\alpha(0,T;X),
\end{equation}
where $0<\alpha<1$, $f\in X$ is fixed, and the scalar-valued temporal
factor $p$ is unknown. The inverse problem is to determine the
fractional order $\alpha$ from the scalar observation
\begin{equation}\label{eq:observation}
\mathcal{M}_{p,\alpha}(t)
:=
\langle u_{p,\alpha}(t),\psi\rangle_X,
\qquad
0<t<T,
\end{equation}
where $\psi\in\mathcal{D}(A)$ is fixed and $u_{p,\alpha}$ denotes the
solution of \eqref{eq:forward}.

The main difficulty is that both the fractional order $\alpha$ and the
temporal factor $p$ are unknown. The condition
\[
p(0)=p_0\ne0
\]
plays a crucial role in separating their leading-order effects near
$t=0$. More precisely, for $0\leq\gamma<\alpha$, we prove the uniform
small-time expansion
\begin{equation}\label{eq:intro-expansion}
\partial_t^\gamma\mathcal{M}_{p,\alpha}(t)
=
p_0\langle f,\psi\rangle_X
\frac{t^{\alpha-\gamma}}
{\Gamma(1+\alpha-\gamma)}
+
O(t^{2\alpha-\gamma}),
\qquad
t\downarrow0.
\end{equation}
The remainder in \eqref{eq:intro-expansion} is uniform with respect to
all admissible temporal factors $p$ and fractional orders $\alpha$.
This expansion provides the key ingredient for the stability analysis:
by comparing the leading profiles corresponding to two different
fractional orders and choosing the observation time appropriately in
terms of the data discrepancy, we derive a 
Hölder stability estimate for the fractional order.

A second point concerns the order of differentiation used in the data norm.
If the observation is differentiated to an order $\rho$ larger than every
admissible fractional order, then, as it is shown in Theorem \ref{thm:rigidity}, the resulting quantity is generally not
bounded near $t=0$.  

Several types of overdetermination conditions have been studied for
the recovery of unknown parameters in time-fractional diffusion
equations; see, for example, the survey paper \cite{LiYamamoto1} and
the references therein. In the present work, we consider the scalar
observation
\[
\mathcal{M}_{p,\alpha}(t)
=
\langle u_{p,\alpha}(t),\psi\rangle_X,
\]
which can be interpreted as a weighted average of the solution of the
corresponding forward problem.

Observations of this type have been investigated in several settings.
In \cite{AU2020}, the weight function $\psi$ is chosen to be the first
eigenfunction of the operator $A$. In \cite{PskhuUzmat}, the problem is
considered in the whole space $\mathbb{R}^N$ with
$A=-\Delta$, while $\psi$ is allowed to be an arbitrary harmonic
function satisfying a Tikhonov-type condition.

Of particular relevance to the present work is \cite{AshurovYamamoto},
where the same type of scalar observation is considered with an
arbitrary weight function $\psi\in X$. In that work, the authors
established the uniqueness of both the fractional order $\alpha$ and
the temporal factor $p$ appearing in the source term. The present
paper develops this direction further by establishing a uniform
small-time asymptotic expansion and deriving a quantitative
stability estimate for the fractional order. We also
clarify the fundamentally different behavior of supercritical
fractional derivatives by proving the rigidity result described above.

The paper is organized as follows. Section~\ref{sec:setting} introduces
the operator framework and derives a representation formula.
In Section~\ref{sec:asymptotics}, we prove the uniform small-time
expansion. Section~\ref{sec:stability} contains the 
stability theorem, while Section~\ref{sec:rigidity} is devoted to the
rigidity phenomenon for supercritical derivatives. The uniform
Mittag--Leffler estimate used in the analysis is proved in
Appendix~\ref{sec:appendix}.
\section{Setting and representation formula}
\label{sec:setting}

We impose the following standard spectral assumption.

\medskip
\noindent
\textbf{Assumption on $A$.}
The operator $A\colon\D(A)\subset X\to X$ is densely defined,
self-adjoint and negative, $-A\ge\lambda_1 I$ for some $\lambda_1>0$, and
$A$ has compact resolvent.  Thus
\begin{equation}
 -A=\sum_{n=1}^\infty\lambda_nP_n,
 \qquad 0<\lambda_1<\lambda_2<\cdots,
 \label{eq:spectral}
\end{equation}
where $P_n$ are the mutually orthogonal spectral projections.  Repeated
eigenvalues are included in the ranges of the projections $P_n$.

For $a,b>0$, the two-parameter Mittag--Leffler function is
\begin{equation}
 E_{a,b}(z):=\sum_{k=0}^\infty\frac{z^k}{\Gamma(ak+b)}.
 \label{eq:ML-definition}
\end{equation}
The solution of \eqref{eq:forward} has the spectral representation
\begin{equation}
 u_{p,\alpha}(t)
 =\sum_{n=1}^\infty\int_0^t
 (t-s)^{\alpha-1}E_{\alpha,\alpha}
 \bigl(-\lambda_n(t-s)^\alpha\bigr)p(s)P_nf\dd s,
 \label{eq:mild-solution}
\end{equation}
where the series is convergent in $H_{\alpha}(0, T; X) \cap L^{2}(0, T ; \mathcal{D}(A))$.
This is the standard mild solution of the time-fractional evolution equation;
see, for example, \cite{KubicaRyszewskaYamamoto,7}.

Fix constants
\begin{equation}
 0<\alpha_-<\alpha_+<1,
 \qquad p_0\ne0,
 \qquad M>\abs{p_0},
 \label{eq:parameters}
\end{equation}
and define
\begin{equation}
 \cA:=[\alpha_-,\alpha_+],
 \qquad
 \cP:=\left\{p\in C^1[0,T]:p(0)=p_0,
 \ \norm{p}_{C^1[0,T]}\le M\right\}.
 \label{eq:admissible-sets}
\end{equation}
Throughout the paper we assume
\begin{equation}
 f\in X,\qquad \psi\in\D(A),\qquad
 c_*:=p_0\ip{f}{\psi}\ne0.
 \label{eq:nondegeneracy}
\end{equation}
The last condition is the natural nondegeneracy of the leading small-time
coefficient.

We first record a representation that separates the leading source term from
the contribution of the operator $A$.

\begin{proposition}
\label{prop:representation}
Let $0\le\gamma<\alpha_-$.  For every $p\in\cP$, $\alpha\in\cA$, and
$t\in(0,T)$,
\begin{equation}
 \partial_t^\gamma u_{p,\alpha}(t)
 = (J^{\alpha-\gamma}p)(t)f
 -\int_0^t R_{\alpha,\gamma}(s)p(t-s)\dd s,
 \label{eq:representation}
\end{equation}
where
\begin{equation}
 R_{\alpha,\gamma}(t)
 :=\sum_{n=1}^\infty
 \lambda_n t^{2\alpha-\gamma-1}
 E_{\alpha,2\alpha-\gamma}(-\lambda_nt^\alpha)P_nf.
 \label{eq:R-kernel}
\end{equation}
Moreover, there is a constant $C>0$, independent of $p$, $\alpha$, $\gamma$, and $t$,
such that
\begin{equation}
 \abs{\ip{R_{\alpha,\gamma}(t)}{\psi}}
 \le C t^{2\alpha-\gamma-1}\norm{f}_X\norm{A\psi}_X,
 \qquad 0<t<T.
 \label{eq:R-scalar-bound}
\end{equation}
\end{proposition}

A detailed proof of the representation (\ref{eq:representation}) - (\ref{eq:R-kernel}), based on the properties of the Mittag-Leffler function, is contained in the paper \cite{AshurovYamamoto} (by the authors of this article). For the reader's convenience, we present here the main points of the proof. The estimate (\ref{eq:R-scalar-bound}) was also proved in the same paper, but without using Lemma \ref{lem:ML-uniform}, and therefore the constant $C$ in that paper depended on $\alpha$ and $\gamma$.

\begin{proof} We first prove the representation formula. Using the identity
\begin{equation}\label{eq2.4}
E_{\alpha,\alpha}(z)
=
\frac{1}{\Gamma(\alpha)}
+
zE_{\alpha,2\alpha}(z),
\qquad z\in\mathbb C,
\end{equation}
which follows directly from the definition of the Mittag--Leffler
function, we rewrite \eqref{eq:mild-solution} as
\begin{equation}\label{eq2.6}
u_{p,\alpha}(t)
=
(J^\alpha p)(t)f
-
\sum_{n=1}^{\infty}
\lambda_n
\left(
t^{2\alpha-1}
E_{\alpha,2\alpha}(-\lambda_n t^\alpha)*p
\right)(t)P_nf.
\end{equation}

Since $0\leq\gamma<\alpha$, the semigroup property of the
Riemann--Liouville integrals yields
\[
\partial_t^\gamma J^\alpha p
=
(J^\gamma)^{-1}J^\alpha p
=
J^{\alpha-\gamma}p.
\]
Moreover, the standard identity for the fractional integral of the
Mittag--Leffler kernel gives (e.g., formula (1.100) on p. 25 in \cite{Podlubny})
\[
J^\gamma
\left(
t^{2\alpha-\gamma-1}
E_{\alpha,2\alpha-\gamma}(-\lambda_n t^\alpha)
\right)
=
t^{2\alpha-1}
E_{\alpha,2\alpha}(-\lambda_n t^\alpha),
\]
and hence
\[
\partial_t^\gamma
\left(
t^{2\alpha-1}
E_{\alpha,2\alpha}(-\lambda_n t^\alpha)
\right)
=
t^{2\alpha-\gamma-1}
E_{\alpha,2\alpha-\gamma}(-\lambda_n t^\alpha).
\]
Therefore, applying $\partial_t^\gamma$ to \eqref{eq2.6} and using
the corresponding convolution identity, we obtain
\[
\partial_t^\gamma u_{p,\alpha}(t)
=
(J^{\alpha-\gamma}p)(t)f
-
\sum_{n=1}^{\infty}
\lambda_n
\left(
t^{2\alpha-\gamma-1}
E_{\alpha,2\alpha-\gamma}(-\lambda_n t^\alpha)*p
\right)(t)P_nf.
\]
Thus,
\[
\partial_t^\gamma u_{p,\alpha}(t)
=
(J^{\alpha-\gamma}p)(t)f
-
\int_0^t R_{\alpha,\gamma}(s)p(t-s)\,ds,
\]
where
\[
R_{\alpha,\gamma}(t)
=
\sum_{n=1}^{\infty}
\lambda_n t^{2\alpha-\gamma-1}
E_{\alpha,2\alpha-\gamma}(-\lambda_n t^\alpha)P_nf.
\]
This proves \eqref{eq:representation}--\eqref{eq:R-kernel}.

It remains to prove the estimate (\ref{eq:R-scalar-bound}). Taking the scalar product with
$\psi$ and using the definition of $R_{\alpha,\gamma}$, we obtain
\[
\left|
\left(R_{\alpha,\gamma}(t),\psi\right)_X
\right|
\leq
\sum_{n=1}^{\infty}
\lambda_nt^{2\alpha-\gamma-1}
\left|
E_{\alpha,2\alpha-\gamma}
(-\lambda_nt^\alpha)
\right|
\left|
(P_nf,\psi)_X
\right|.
\]
By Lemma \ref{lem:ML-uniform} (see Appendix \ref{sec:appendix}), since
\[
\alpha\in[\alpha_-,\alpha_+],
\qquad
2\alpha-\gamma
\in
[2\alpha_- -\gamma,\,2\alpha_+-\gamma],
\]
there is a constant $C>0$, independent of $\alpha$, $\gamma$, $n$ and $t$, such
that
\[
\left|
E_{\alpha,2\alpha-\gamma}
(-\lambda_nt^\alpha)
\right|
\leq
\frac{C}
{1+\lambda_nt^\alpha}.
\]
Therefore,
\[
\begin{aligned}
\left|
\left(R_{\alpha,\gamma}(t),\psi\right)_X
\right|
&\leq
Ct^{2\alpha-\gamma-1}
\sum_{n=1}^{\infty}
\left|
\left(P_nf,\lambda_nP_n\psi\right)_X
\right|\\
&\leq
Ct^{2\alpha-\gamma-1}
\left(
\sum_{n=1}^{\infty}\|P_nf\|_X^2
\right)^{1/2}
\left(
\sum_{n=1}^{\infty}
\|\lambda_nP_n\psi\|_X^2
\right)^{1/2}\\
&=
Ct^{2\alpha-\gamma-1}
\|f\|_X\|A\psi\|_X.
\end{aligned}
\]
Thus
\[
\left|
\left(R_{\alpha,\gamma}(t),\psi\right)_X
\right|
\leq
Ct^{2\alpha-\gamma-1}
\|f\|_X\|A\psi\|_X,
\qquad 0<t<T.
\]
The constant $C$ is independent of $p$, $\alpha$, $\gamma$, and $t$. This proves
(\ref{eq:R-scalar-bound}) and completes the proof.
\end{proof}

\section{Uniform small-time asymptotics}
\label{sec:asymptotics}

The following expansion is the main analytic input for both the stability and
rigidity results.

\begin{proposition}
\label{prop:small-time}
Fix $0\le\gamma<\alpha_-$.  There exist $t_1\in(0,\min\{1,T\}]$ and $C>0$
such that
\begin{equation}
 \left|
 \partial_t^\gamma\cM_{p,\alpha}(t)
 -c_*\frac{t^{\alpha-\gamma}}
 {\Gamma(1+\alpha-\gamma)}
 \right|
 \le C t^{2\alpha-\gamma}
 \label{eq:uniform-expansion}
\end{equation}
for all $p\in\cP$, $\alpha\in\cA$, and $0<t\le t_1$.
The constant $C$ depends only on the fixed a priori data
\begin{equation*}
 T,\ \alpha_-,\ \alpha_+,\ \gamma,\ M,\ p_0,\ f,\ \psi,
 \quad\text{and }A.
\end{equation*}
\end{proposition}

\begin{proof}
Taking the scalar product of \eqref{eq:representation} with $\psi$ gives
\begin{equation}
 \partial_t^\gamma\cM_{p,\alpha}(t)
 =\ip{f}{\psi}(J^{\alpha-\gamma}p)(t)
 -\int_0^t\ip{R_{\alpha,\gamma}(s)}{\psi}p(t-s)\dd s.
 \label{eq:scalar-representation}
\end{equation}
Since $p(0)=p_0$ and $\norm{p}_{C^1}\le M$,
\begin{align}
 \left|(J^{\alpha-\gamma}p)(t)
 -p_0\frac{t^{\alpha-\gamma}}
 {\Gamma(1+\alpha-\gamma)}\right|
 &\le \frac{M}{\Gamma(\alpha-\gamma)}
 \int_0^t(t-s)^{\alpha-\gamma-1}s\dd s \notag\\
 &\le C t^{\alpha-\gamma+1}.
 \label{eq:J-expansion}
\end{align}
On the other hand, \eqref{eq:R-scalar-bound} and
$\norm{p}_{C[0,T]}\le M$ imply
\begin{equation}
 \left|\int_0^t\ip{R_{\alpha,\gamma}(s)}{\psi}p(t-s)\dd s\right|
 \le C\int_0^t s^{2\alpha-\gamma-1}\dd s
 \le C t^{2\alpha-\gamma}.
 \label{eq:R-convolution-bound}
\end{equation}
Because $\alpha\le\alpha_+<1$ and $0<t\le1$,
$t^{\alpha-\gamma+1}\le t^{2\alpha-\gamma}$.  Combining
\eqref{eq:scalar-representation}--\eqref{eq:R-convolution-bound} proves
\eqref{eq:uniform-expansion}.
\end{proof}

\begin{remark}
The common value $p(0)=p_0$ is what makes the leading coefficient in
\eqref{eq:uniform-expansion} independent of the unknown function $p$.  If
$p_0=0$ or $\ip{f}{\psi}=0$, the leading term vanishes and a different
nondegeneracy condition, involving a higher-order source jet or a higher
spectral moment, is required.
\end{remark}

\section{Stability for subcritical observations}
\label{sec:stability}

We begin with an elementary inversion estimate for the leading profile.
Let $\Psi=\Gamma'/\Gamma$ denote the digamma function.

\begin{lemma}\label{lem:profile-inversion}
Let \(0<r_-<r_+\) and \(\eta_+>0\). Then there exist constants
\(t_0\in(0,1)\), \(\varepsilon_0>0\), and \(C>0\), depending only on
\(r_-,r_+\), and \(\eta_+\), such that the following assertion holds.

Suppose that
$$
 r\in[r_-,r_+],\qquad
 \eta\in[0,\eta_+],\qquad
 0<t\leq t_0,
$$
and
$$
\left|
\frac{1}{\Gamma(1+r)}
-
\frac{t^\eta}{\Gamma(1+r+\eta)}
\right|
\leq \varepsilon
\leq \varepsilon_0.
$$
Then
$$
\eta\leq
C\frac{\varepsilon}{1+|\log t|}.
$$

\end{lemma}

\begin{proof}
For fixed \(t\in(0,1)\) and \(r\in[r_-,r_+]\), define
$$
H_{r,t}(s):=
\frac{t^s}{\Gamma(1+r+s)},
\qquad 0\leq s\leq\eta_+.
$$
Then
$$
H_{r,t}(0)=\frac{1}{\Gamma(1+r)},
\qquad
H_{r,t}(\eta)
=
\frac{t^\eta}{\Gamma(1+r+\eta)}.
$$
Consequently, the assumption of the lemma can be written as
$$
|H_{r,t}(0)-H_{r,t}(\eta)|\leq\varepsilon.
$$
We first establish a uniform lower bound for \(H_{r,t}\) on the
interval \([0,\eta]\), provided \(\varepsilon\) is sufficiently small.

Since
$$
1+r+s\in[1+r_-,1+r_++\eta_+],
$$
and \(\Psi\) is continuous on this compact interval, there exists a
constant \(L>0\) such that
$$
|\Psi(1+r+s)|\leq L
$$
for all
$$
r\in[r_-,r_+],
\qquad
s\in[0,\eta_+].
$$
Differentiating \(H_{r,t}\) with respect to \(s\), we obtain
$$
H_{r,t}'(s)
=
H_{r,t}(s)
\left(
\log t-\Psi(1+r+s)
\right).
$$
Since \(0<t<1\), we have \(\log t<0\). Choose \(t_0\in(0,1)\)
so small that
$$
|\log t|\geq 2L+2,
\qquad 0<t\leq t_0.
$$
It follows that
$$
\log t-\Psi(1+r+s)
\leq
-|\log t|+L
\leq
-\frac12|\log t|-1
<0.
$$
Hence
$$
H_{r,t}'(s)<0,
\qquad
0\leq s\leq\eta_+,
$$
and therefore \(H_{r,t}\) is strictly decreasing on
\([0,\eta_+]\).
Next, since
$$
r\in[r_-,r_+],
$$
the continuity and positivity of the Gamma function imply
$$
m_0:=
\min_{r\in[r_-,r_+]}
\frac{1}{\Gamma(1+r)}
>0.
$$
Choose
$$
\varepsilon_0\leq\frac{m_0}{2}.
$$
From
$$
|H_{r,t}(0)-H_{r,t}(\eta)|\leq\varepsilon
\leq\varepsilon_0
$$
we obtain
$$
H_{r,t}(\eta)
\geq
H_{r,t}(0)-\varepsilon_0
\geq
m_0-\frac{m_0}{2}
=
\frac{m_0}{2}.
$$
Since \(H_{r,t}\) is decreasing, this yields
$$
H_{r,t}(s)\geq\frac{m_0}{2},
\qquad
0\leq s\leq\eta.
$$
We can therefore estimate the derivative from below. Indeed,
$$
-H_{r,t}'(s)
=
H_{r,t}(s)
\left(
|\log t|+\Psi(1+r+s)
\right),
$$
where we have used
$$
-\log t=|\log t|.
$$
Since
$$
\Psi(1+r+s)\geq-L,
$$
we have
$$
|\log t|+\Psi(1+r+s)
\geq
|\log t|-L.
$$
Thus
$$
-H_{r,t}'(s)
\geq
\frac{m_0}{2}
\left(
|\log t|-L
\right).
$$
By the choice of \(t_0\), we have
$$
|\log t|-L
\geq
\frac12(1+|\log t|),
\qquad
0<t\leq t_0.
$$
Consequently, there exists a constant \(c>0\), depending only on
\(r_-,r_+\), and \(\eta_+\), such that
$$
-H_{r,t}'(s)
\geq
c(1+|\log t|),
\qquad
0\leq s\leq\eta.
$$
Integrating this inequality over \([0,\eta]\), we obtain
$$
H_{r,t}(0)-H_{r,t}(\eta)
=
-\int_0^\eta H_{r,t}'(s)\,ds
\geq
c\eta(1+|\log t|).
$$
On the other hand, by assumption,
$$
H_{r,t}(0)-H_{r,t}(\eta)
\leq
|H_{r,t}(0)-H_{r,t}(\eta)|
\leq\varepsilon.
$$
Hence
$$
c\eta(1+|\log t|)
\leq\varepsilon,
$$
and therefore
$$
\eta
\leq
\frac{1}{c}
\frac{\varepsilon}{1+|\log t|}.
$$
This proves the desired estimate.
\end{proof}

We now state the main stability theorem.  For $0\le\gamma<\alpha_-$, define
\begin{equation}
 D_\gamma(p,\alpha;q,\beta)
 :=\norm{\partial_t^\gamma\cM_{p,\alpha}
 -\partial_t^\gamma\cM_{q,\beta}}_{L^\infty(0,T)}.
 \label{eq:data-discrepancy}
\end{equation}
The derivatives in \eqref{eq:data-discrepancy} are bounded because
$\gamma$ lies strictly below every admissible fractional order.

\begin{theorem}
\label{thm:stability}
Fix $0\le\gamma<\alpha_-$ and set
\begin{equation}
 \theta_\gamma:=\frac{\alpha_+}{2\alpha_+-\gamma}\in\left[\frac12,1\right).
 \label{eq:theta}
\end{equation}
There exist constants $D_*>0$ and $C>0$ such that, for all
$p,q\in\cP$ and $\alpha,\beta\in\cA$, the following holds.  If
\begin{equation*}
 0<D_\gamma:=D_\gamma(p,\alpha;q,\beta)\le D_*,
\end{equation*}
then
\begin{equation}\label{eq:main-stability}
 \abs{\alpha-\beta}
 \le C\frac{D_\gamma^{\theta_\gamma}}
 {1+\abs{\log D_\gamma}}.
\end{equation}
If $D_\gamma=0$, then $\alpha=\beta$.
\end{theorem}

\begin{proof}
By symmetry with respect to the pairs \((p,\alpha)\) and
\((q,\beta)\), it is sufficient to consider the case
$$
\alpha\geq\beta.
$$
Set
$$
\eta:=\alpha-\beta\geq0,
\qquad
r:=\beta-\gamma.
$$
Since
$$
\beta\in[\alpha_-,\alpha_+]
\qquad\text{and}\qquad
0\leq\gamma<\alpha_-,
$$
we have
$$
r\in[\alpha_--\gamma,\alpha_+-\gamma].
$$
In particular, the parameter \(r\) varies in a compact interval
strictly contained in \((0,\infty)\). Moreover,
$$
0\leq\eta=\alpha-\beta\leq\alpha_+-\alpha_-.
$$
Hence the parameters \(r\) and \(\eta\) satisfy the assumptions of
Lemma~\ref{lem:profile-inversion} with
$$
r_-=\alpha_--\gamma,\qquad
r_+=\alpha_+-\gamma,\qquad
\eta_+=\alpha_+-\alpha_-.
$$

By Proposition~\ref{prop:small-time}, there exist \(t_1\in(0,\min\{1,T\}]\) and
\(C>0\), depending only on the fixed a priori data, such that
$$
\left|
\partial_t^\gamma \mathcal{M}_{p,\alpha}(t)
-
c_*
\frac{t^{\alpha-\gamma}}
{\Gamma(1+\alpha-\gamma)}
\right|
\leq
Ct^{2\alpha-\gamma},
\qquad
0<t\leq t_1,
$$
and, similarly,
$$
\left|
\partial_t^\gamma \mathcal{M}_{q,\beta}(t)
-
c_*
\frac{t^{\beta-\gamma}}
{\Gamma(1+\beta-\gamma)}
\right|
\leq
Ct^{2\beta-\gamma}.
$$
The same constant \(C\) may be used for both estimates because the
remainder estimate in Proposition~\ref{prop:small-time} is uniform with respect to all
admissible \(p,q\in\mathcal P\) and
\(\alpha,\beta\in\mathcal A\).

Subtracting the two asymptotic formulas and using the triangle
inequality, we obtain
$$
\begin{aligned}
|c_*|
\left|
\frac{t^{\beta-\gamma}}
{\Gamma(1+\beta-\gamma)}
-
\frac{t^{\alpha-\gamma}}
{\Gamma(1+\alpha-\gamma)}
\right|
&\leq
\left|
\partial_t^\gamma \mathcal{M}_{p,\alpha}(t)
-
\partial_t^\gamma \mathcal{M}_{q,\beta}(t)
\right|
\\
&\quad
+Ct^{2\beta-\gamma}
+Ct^{2\alpha-\gamma}.
\end{aligned}
$$

By the definition of \(D_\gamma\),
$$
\left|
\partial_t^\gamma \mathcal{M}_{p,\alpha}(t)
-
\partial_t^\gamma \mathcal{M}_{q,\beta}(t)
\right|
\leq D_\gamma,
$$
and hence
$$
|c_*|
\left|
\frac{t^{\beta-\gamma}}
{\Gamma(1+\beta-\gamma)}
-
\frac{t^{\alpha-\gamma}}
{\Gamma(1+\alpha-\gamma)}
\right|
\leq
D_\gamma
+Ct^{2\beta-\gamma}
+Ct^{2\alpha-\gamma}.
$$
Since \(\alpha\geq\beta\) and \(0<t\leq1\), then
$
t^{2\alpha-\gamma}
\leq
t^{2\beta-\gamma}.
$
Therefore,
$$
|c_*|
\left|
\frac{t^{\beta-\gamma}}
{\Gamma(1+\beta-\gamma)}
-
\frac{t^{\alpha-\gamma}}
{\Gamma(1+\alpha-\gamma)}
\right|
\leq
D_\gamma
+Ct^{2\beta-\gamma}.
$$
After enlarging \(C\), if necessary, we can write
$$
\left|
\frac{t^{\beta-\gamma}}
{\Gamma(1+\beta-\gamma)}
-
\frac{t^{\alpha-\gamma}}
{\Gamma(1+\alpha-\gamma)}
\right|
\leq
C\left(
D_\gamma+t^{2\beta-\gamma}
\right).
$$

We now factor out \(t^{\beta-\gamma}\). Since
$$
\alpha-\gamma
=
(\beta-\gamma)+(\alpha-\beta)
=
r+\eta,
$$
we have
$$
t^{\alpha-\gamma}
=
t^{\beta-\gamma}t^\eta.
$$
Consequently,
$$
t^{\beta-\gamma}
\left|
\frac{1}{\Gamma(1+r)}
-
\frac{t^\eta}{\Gamma(1+r+\eta)}
\right|
\leq
C\left(
D_\gamma+t^{2\beta-\gamma}
\right).
$$
Dividing by \(t^{\beta-\gamma}=t^r\), we obtain
\begin{equation}\label{eq:key-stability-inequality}
\left|
\frac{1}{\Gamma(1+r)}
-
\frac{t^\eta}{\Gamma(1+r+\eta)}
\right|
\leq
C\left(
D_\gamma t^{-r}+t^\beta
\right).
\end{equation}

We now choose \(t\) in such a way that the two terms on the
right-hand side have the same order. Let
\begin{equation}\label{eq:choice-t}
    t:=D_\gamma^{1/(2\beta-\gamma)}.
\end{equation}
Since
$$
2\beta-\gamma
\geq
2\alpha_- -\gamma>0,
$$
this is well defined.

Assume first that
$$
0<D_\gamma\leq D_*,
$$
where \(D_*>0\) will be chosen sufficiently small below. Then
\(t\to0\) uniformly with respect to
\(\beta\in[\alpha_-,\alpha_+]\) as \(D_\gamma\to0\).
Indeed,
$$
t
=
D_\gamma^{1/(2\beta-\gamma)}
\leq
D_\gamma^{1/(2\alpha_+-\gamma)}
$$
for \(0<D_\gamma<1\). Hence, by reducing \(D_*\), if necessary, we
may guarantee
$$
0<t\leq\min\{t_0,t_1\},
$$
where \(t_0\) is the constant appearing in Lemma~\ref{lem:profile-inversion} and \(t_1\)
comes from Proposition~\ref{prop:small-time}.

With the choice (\ref{eq:choice-t}), we have
$$
t^{2\beta-\gamma}=D_\gamma.
$$
Since
$$
r=\beta-\gamma,
$$
it follows that
$$
D_\gamma t^{-r}
=
D_\gamma
D_\gamma^{-(\beta-\gamma)/(2\beta-\gamma)}
=
D_\gamma^{\beta/(2\beta-\gamma)}.
$$

On the other hand,
$$
t^\beta
=
D_\gamma^{\beta/(2\beta-\gamma)}.
$$
Thus the two terms are exactly equal:
$$
D_\gamma t^{-r}
=
t^\beta
=
D_\gamma^{\beta/(2\beta-\gamma)}.
$$
Therefore, from (\ref{eq:key-stability-inequality}),
\begin{equation}\label{Dbeta}
    \left|
\frac{1}{\Gamma(1+r)}
-
\frac{t^\eta}{\Gamma(1+r+\eta)}
\right|
\leq
C
D_\gamma^{\beta/(2\beta-\gamma)}.
\end{equation}

We next verify that Lemma~\ref{lem:profile-inversion} can indeed be applied. Since
$$
\frac{\beta}{2\beta-\gamma}>0
$$
uniformly for
\(\beta\in[\alpha_-,\alpha_+]\), the right-hand side of
(\ref{Dbeta}) tends to zero uniformly as \(D_\gamma\to0\).
Thus, after decreasing \(D_*\) once more, we may assume that
$$
C D_\gamma^{\beta/(2\beta-\gamma)}
\leq\varepsilon_0,
$$
where \(\varepsilon_0\) is the constant of Lemma~\ref{lem:profile-inversion}.

Applying Lemma~\ref{lem:profile-inversion} with
$$
\varepsilon
=
C D_\gamma^{\beta/(2\beta-\gamma)},
$$
we obtain
\begin{equation}\label{eta1}
\eta
\leq
C
\frac{
D_\gamma^{\beta/(2\beta-\gamma)}
}{
1+|\log t|
}.
\end{equation}

Recall that \(\eta=\alpha-\beta\), and therefore
$$
|\alpha-\beta|
=
\eta.
$$

It remains to replace the exponent
$$
\frac{\beta}{2\beta-\gamma}
$$
by the uniform exponent \(\theta_\gamma\).
Consider the function
$$
F(s):=\frac{s}{2s-\gamma},
\qquad s>\frac{\gamma}{2}
$$
(note, in our case even $s>\gamma$). Its derivative is
$$
F'(s)
=
\frac{(2s-\gamma)-2s}{(2s-\gamma)^2}
=
-\frac{\gamma}{(2s-\gamma)^2}
\leq0.
$$

Thus \(F\) is nonincreasing. Since
$$
\beta\leq\alpha_+,
$$
we obtain
$$
\frac{\beta}{2\beta-\gamma}
\geq
\frac{\alpha_+}{2\alpha_+-\gamma}
=
\theta_\gamma.
$$

For \(0<D_\gamma<1\), a larger exponent gives a smaller power, and
hence
$$
D_\gamma^{\beta/(2\beta-\gamma)}
\leq
D_\gamma^{\theta_\gamma}.
$$

Furthermore, by the definition of \(t\),
$$
|\log t|
=
\frac{|\log D_\gamma|}{2\beta-\gamma}.
$$
Since
$$
2\alpha_- -\gamma
\leq
2\beta-\gamma
\leq
2\alpha_+-\gamma,
$$
there exist constants \(c_1,c_2>0\), independent of \(\beta\), such
that
$$
c_1(1+|\log D_\gamma|)
\leq
1+|\log t|
\leq
c_2(1+|\log D_\gamma|).
$$

In particular,
$$
\frac{1}{1+|\log t|}
\leq
\frac{C}{1+|\log D_\gamma|}.
$$

Combining this estimate with (\ref{eta1}), we conclude that
$$
|\alpha-\beta|
\leq
C
\frac{
D_\gamma^{\theta_\gamma}
}{
1+|\log D_\gamma|
},
$$
which proves (\ref{eq:main-stability}).

It remains to consider the case \(D_\gamma=0\).
Assume, to the contrary, that
$$
\eta=\alpha-\beta>0.
$$
Then (\ref{eq:key-stability-inequality}) gives
$$
\left|
\frac{1}{\Gamma(1+r)}
-
\frac{t^\eta}{\Gamma(1+r+\eta)}
\right|
\leq
Ct^\beta,
\qquad
0<t\leq t_1.
$$

Since \(\eta>0\) and \(\beta>0\),
$$
t^\eta\to0,
\qquad
t^\beta\to0
\qquad
\text{as }t\downarrow0.
$$
Passing to the limit \(t\downarrow0\), we obtain
$$
\frac{1}{\Gamma(1+r)}=0.
$$

This is impossible because \(r>0\), and hence
$$
\Gamma(1+r)>0.
$$

Therefore \(\eta=0\), so that
$$
\alpha=\beta.
$$

This completes the proof.
\end{proof}

\begin{corollary}
\label{cor:gamma-zero}
For
\begin{equation*}
 D_0:=\norm{\cM_{p,\alpha}-\cM_{q,\beta}}_{L^\infty(0,T)},
\end{equation*}
there exist $D_*>0$ and $C>0$ such that
\begin{equation}
 \abs{\alpha-\beta}
 \le C\frac{D_0^{1/2}}{1+\abs{\log D_0}}
 \label{eq:gamma-zero-stability}
\end{equation}
whenever $0<D_0\le D_*$.  In particular, the order is uniquely determined
within the admissible class by the scalar observation \eqref{eq:observation},
even though the temporal factor is also unknown.
\end{corollary}

\begin{remark}
The logarithmic denominator in \eqref{eq:main-stability} comes from the
sensitivity of $t^\alpha$ to the exponent:
$\partial_\alpha t^\alpha=t^\alpha\log t$.  Dropping this denominator gives
the weaker but simpler H\"older estimate
$\abs{\alpha-\beta}\le C D_\gamma^{\theta_\gamma}$.
\end{remark}

\section{Supercritical derivatives and rigidity}
\label{sec:rigidity}

The stability norm in Theorem~\ref{thm:stability} uses
$\gamma<\alpha_-$.  We now explain why a derivative order larger than all
admissible orders behaves very differently. Note that in case $X= \mathbb{R}$ we will additionally introduce the notation $H_{\alpha}(0,T) := H_{\alpha}(0,T;\mathbb{R})$.

\begin{theorem}\label{thm:rigidity} Let $\rho\in(\alpha_+,1)$, $p,q\in\mathcal P$, and
$\alpha,\beta\in\mathcal A$. Set
\[
w(t):=\mathcal{M}_{p,\alpha}(t)-\mathcal{M}_{q,\beta}(t).
\]
If
\[
w\in H_\rho(0,T)
\quad\text{and}\quad
\partial_t^\rho w\in L^\infty(0,T),
\]
then
\[
\alpha=\beta.
\]
Equivalently, if $\alpha\ne\beta$, then
\[
\partial_t^\rho
\bigl(\mathcal{M}_{p,\alpha}-\mathcal{M}_{q,\beta}\bigr)
\notin L^\infty(0,T).
\]
In particular,
\[
\left\|
\partial_t^\rho
\bigl(\mathcal{M}_{p,\alpha}-\mathcal{M}_{q,\beta}\bigr)
\right\|_{L^\infty(0,T)}
=\infty
\]
whenever $\alpha\ne\beta$.
\end{theorem}

\begin{proof}
We argue by contradiction. Suppose that
\[
\alpha\ne\beta.
\]
Without loss of generality, we may assume that
\[
\alpha<\beta.
\]

By Proposition~\ref{prop:small-time} with $\gamma=0$, uniformly with respect to
$p,q\in\mathcal P$ and $\alpha,\beta\in\mathcal A$, we have
\[
\mathcal{M}_{p,\alpha}(t)
=
c_*
\frac{t^\alpha}{\Gamma(1+\alpha)}
+
O(t^{2\alpha}),
\qquad t\downarrow0,
\]
and
\[
\mathcal{M}_{q,\beta}(t)
=
c_*
\frac{t^\beta}{\Gamma(1+\beta)}
+
O(t^{2\beta}),
\qquad t\downarrow0,
\]
where
\[
c_*=p_0(f,\psi)_X\ne0.
\]
Consequently,
\[
\begin{aligned}
w(t)
&=
\mathcal{M}_{p,\alpha}(t)-\mathcal{M}_{q,\beta}(t)
\\
&=
c_*
\frac{t^\alpha}{\Gamma(1+\alpha)}
-
c_*
\frac{t^\beta}{\Gamma(1+\beta)}
+
O(t^{2\alpha})
+
O(t^{2\beta}).
\end{aligned}
\]
Since $\alpha<\beta$, we have
\[
t^\beta=o(t^\alpha),
\qquad
t^{2\alpha}=o(t^\alpha),
\qquad
t^{2\beta}=o(t^\alpha)
\]
as $t\downarrow0$. Hence
\[
w(t)
=
c_*
\frac{t^\alpha}{\Gamma(1+\alpha)}
+
o(t^\alpha),
\qquad t\downarrow0.
\]
Therefore,
\[
\lim_{t\downarrow0}
\frac{w(t)}{t^\alpha}
=
\frac{c_*}{\Gamma(1+\alpha)}
\ne0.
\]
It follows that there exist constants $c>0$ and
$t_2\in(0,T)$ such that
\begin{equation}\label{w-from below}
|w(t)|\geq ct^\alpha,
\qquad 0<t<t_2.
\end{equation}

On the other hand, by assumption,
\[
w\in H_\rho(0,T)
\]
and
\[
g:=\partial_t^\rho w\in L^\infty(0,T).
\]
By the definition of the fractional derivative introduced in
Section~1,
\(
\partial_t^\rho=(J^\rho)^{-1}
\)
with
\(
D(\partial_t^\rho)=H_\rho(0,T).
\)
Hence
\(
w=J^\rho g \in H_\rho(0,T).
\)
Therefore,
\[
w(t)
=
\frac{1}{\Gamma(\rho)}
\int_0^t
(t-s)^{\rho-1}g(s)\,ds
\]
for almost every $t\in(0,T)$.

Since $g\in L^\infty(0,T)$, we obtain
\[
\begin{aligned}
|w(t)|
&\leq
\frac{1}{\Gamma(\rho)}
\int_0^t
(t-s)^{\rho-1}|g(s)|\,ds
\leq
\frac{\|g\|_{L^\infty(0,T)}}{\Gamma(\rho)}
\int_0^t
(t-s)^{\rho-1}\,ds
\\
&=
\frac{\|g\|_{L^\infty(0,T)}}{\Gamma(\rho)}
\frac{t^\rho}{\rho}
=
\frac{\|g\|_{L^\infty(0,T)}}{\Gamma(1+\rho)}
t^\rho.
\end{aligned}
\]
Thus
\begin{equation}\label{w-from-above}
|w(t)|
\leq
\frac{\|g\|_{L^\infty(0,T)}}{\Gamma(1+\rho)}
t^\rho
\end{equation}
for almost every $t\in(0,T)$.

Combining (\ref{w-from below}) and (\ref{w-from-above}), we obtain
\[
ct^\alpha
\leq
\frac{\|g\|_{L^\infty(0,T)}}{\Gamma(1+\rho)}
t^\rho
\]
for almost every $t\in(0,t_2)$. Hence
\[
c
\leq
\frac{\|g\|_{L^\infty(0,T)}}{\Gamma(1+\rho)}
t^{\rho-\alpha}.
\]
Since
\(
\rho>\alpha_+\geq\alpha,
\)
we have
\(
\rho-\alpha>0.
\)
Letting $t\downarrow0$ along a sequence of points for which the
above inequality holds, we obtain
\[
c\leq0,
\]
which contradicts $c>0$.

Therefore, the assumption $\alpha\ne\beta$ is impossible. Hence
\[
\alpha=\beta.
\]
This proves the theorem.
\end{proof}
\begin{corollary}
Suppose the admissible interval satisfies
$\alpha_+<1-\delta$ for some $\delta\in(0,1)$.  Then
\begin{equation*}
 \partial_t^{1-\delta}
 \bigl(\cM_{p,\alpha}-\cM_{q,\beta}\bigr)
 \in L^\infty(0,T)
 \quad\Longrightarrow\quad
 \alpha=\beta.
\end{equation*}
Thus the $L^\infty$ norm of this derivative is not a finite metric for
quantitative stability between different orders; it is a rigidity test.
\end{corollary}

\begin{remark}
Theorem~\ref{thm:rigidity} does not say that every supercritical derivative is
undefined pointwise.  Rather, it says that under the nondegeneracy condition
\eqref{eq:nondegeneracy}, the derivative of the difference cannot remain
bounded up to $t=0$ when the orders differ.  Weighted norms or observations
away from $t=0$ lead to different inverse problems.
\end{remark}

\section{Conclusion}

We studied the inverse problem of recovering the fractional
order in a time-fractional evolution equation with an unknown
temporal source factor. Under the nondegeneracy condition
$p_0\langle f,\psi\rangle_X\neq0$, we established a uniform
small-time asymptotic expansion of the scalar observation. This
expansion yields the  
stability estimate
\[
|\alpha-\beta|
\le
C\frac{D_\gamma^{\theta_\gamma}}
{1+|\log D_\gamma|},
\qquad
\theta_\gamma=
\frac{\alpha_+}{2\alpha_+-\gamma},
\]
for every subcritical derivative order $0\leq\gamma<\alpha_-$.
In particular, the undifferentiated observation determines the
fractional order with a square-root
stability. In contrast, for supercritical differentiation
$\rho>\alpha_+$, boundedness of the fractional derivative of the
difference of two observations forces the two fractional orders
to coincide, giving a rigidity phenomenon rather than a finite
stability estimate. A uniform estimate for the two-parameter
Mittag--Leffler function, including the diagonal case, was also
established and used in the analysis.

\appendix
\section{A uniform Mittag--Leffler estimate}
\label{sec:appendix}

We prove the estimate used in Proposition~\ref{prop:representation}.  The
formulation below is slightly more general than needed.

\begin{lemma}\label{lem:ML-uniform}
Let
\begin{equation*}
 0<a_-<a_+<1,
 \qquad 0<b_-<b_+<\infty.
\end{equation*}
There exists $C>0$ such that
\begin{equation}
 \abs{E_{a,b}(-t)}\le\frac{C}{1+t}
 \label{eq:ML-uniform}
\end{equation}
for all $t\ge0$, $a\in[a_-,a_+]$, and $b\in[b_-,b_+]$.
\end{lemma}

\begin{proof}
We first consider $t\ge1$.  Fix $\varepsilon\in(0,1)$ and set
\begin{equation*}
 \varphi_a:=\frac{\pi a}{2}(1+\varepsilon).
\end{equation*}
Then $\pi a/2<\varphi_a<\pi a<\pi$. Let us define a curve on $\mathbb{C}$ by
\begin{equation*}
\Gamma_{\varphi_a}:= \{re^{-i\varphi_a}:r\ge1\}
 \cup\{e^{i\vartheta}:-\varphi_a\le\vartheta\le\varphi_a\}
 \cup\{re^{i\varphi_a}:r\ge1\} =: \Gamma_1 \cup \Gamma_2 \cup \Gamma_3,
\end{equation*}
which is oriented from $\infty e^{-i\varphi_a}$ to $\infty e^{i\varphi_a}$.
The standard contour representation of the Mittag--Leffler function for $t>0$ gives (see, for example, \cite[Theorem 1.1]{Podlubny}).
\begin{equation}
 E_{a,b}(-t)
 =\frac{1}{2\pi ai}\int_{\Gamma_{\varphi_a}}
 \frac{e^{\zeta^{1/a}}\zeta^{(1-b)/a}}{\zeta+t}\dd\zeta.
 \label{eq:ML-contour}
\end{equation}

We discuss the case $a \ne b$ and $a = b$ 
separately.

\textbf{Case 1:} $a \ne b$. Using
\begin{equation*}
 \frac{1}{\zeta+t}=\frac1t-\frac{\zeta}{t(\zeta+t)}
\end{equation*}
and the Hankel formula for the reciprocal gamma function
\[
\frac{1}{2\pi ai}\int_{\Gamma_{\varphi_a}}
 e^{\zeta^{1/a}}\zeta^{(1-b)/a}\dd\zeta = \frac{1}{\Gamma(b-a)},
\]
we obtain
\begin{equation}
 E_{a,b}(-t)
 =\frac{1}{t\Gamma(b-a)}-\frac{1}{2\pi ait}
 \int_{\Gamma_{\varphi_a}}
 \frac{e^{\zeta^{1/a}}
 \zeta^{(1-b)/a+1}}{\zeta+t}\dd\zeta.
 \label{eq:ML-one-term}
\end{equation}
Here $1/\Gamma(b-a)$ is interpreted by its entire continuation; in particular,
it vanishes when $b-a$ is a nonpositive integer.  Since $b-a$ ranges over a
compact interval,
\begin{equation}
 \sup_{a\in[a_-,a_+],\,b\in[b_-,b_+]}
 \abs{\frac{1}{\Gamma(b-a)}}= C_0 <\infty.
 \label{eq:reciprocal-gamma-bound}
\end{equation}

By the geometry of the contour $\Gamma_{\varphi_a}$, for
$\zeta\in\Gamma_{\varphi_a}$ and $t\geq1$ we have
\[
|\zeta+t|\geq t\sin\varphi_a.
\]
Moreover,
\[
\varphi_a=\frac{\pi a}{2}(1+\varepsilon),
\qquad
a\in[a_-,a_+],
\]
and hence
\[
\frac{\pi a_-}{2}(1+\varepsilon)
\leq \varphi_a
\leq
\frac{\pi a_+}{2}(1+\varepsilon)<\pi.
\]
Since the interval on the right is a compact subinterval of
$(0,\pi)$, we have
\[
c_0:=
\inf_{a\in[a_-,a_+]}\sin\varphi_a>0.
\]
Consequently,
\begin{equation}\label{sin}
|\zeta+t|\geq t\sin\phi_a\geq c_0t,
\qquad
\zeta\in\Gamma_{\phi_a},\quad t\geq1.
\end{equation}

Therefore,
\begin{equation}\label{Estimate_E}
    |E_{a,b}(-t)| \leq \frac{C_0}{t} +
\frac{1}{2\pi a t^2 c_0} \int_{\Gamma_{\varphi_a}} 
 |e^{\zeta^{1/a}}||\zeta^{(1-b)/a+1}||\dd\zeta|.
\end{equation}
Let us denote for $p=1,2$ and $a\in[a_-,a_+]$, and $b\in[b_-,b_+]$ the following integral
\[
    C_{a,b,p}
:= \int_{\Gamma_{\varphi_a}} 
 |e^{\zeta^{1/a}}||\zeta^{(1-b)/a+p}||\dd\zeta|
= \left( \int_{\Gamma_1} + \int_{\Gamma_2} + \int_{\Gamma_3}\right)
 |e^{\zeta^{1/a}}||\zeta^{(1-b)/a+p}||\dd\zeta|.
\] 
Let us estimate \(C_{a,b,p}\).  On either ray,
\begin{equation*}
 \abs{e^{\zeta^{1/a}}}
 =\exp\left(r^{1/a}\cos\frac{\varphi_a}{a}\right)
 =\exp(-c_\varepsilon r^{1/a})
 \le\exp(-c_\varepsilon r^{1/a_+}),
\end{equation*}
where
$c_\varepsilon=-\cos(\pi(1+\varepsilon)/2)>0$.  
Further, we note,
$$
\vert \zeta^{\frac{1-b}{a}+p} \vert 
= r^{\frac{1-b}{a}+p} 
 \quad \mbox{for }
\quad \zeta = re^{\pm i\varphi_a}.
$$
Consequently,
\[
 \sup_{a\in[a_-,a_+],\,b\in[b_-,b_+],\,\, p=0,1}
  \left( \int_{\Gamma_1}  + \int_{\Gamma_3}\right)
 \abs{e^{\zeta^{1/a}}}
 \abs{\zeta^{(1-b)/a+p}}\abs{\dd\zeta}<\infty.
 \label{eq:uniform-contour-integral}
\]
Finally we estimate on $\Gamma_2$.  Let 
$\zeta = e^{i\theta}$ with $-\varphi_a \le \theta \le \varphi_a$. Then
$$
\int_{\Gamma_{2}} 
 |e^{\zeta^{1/a}}||\zeta^{(1-b)/a+p}||\dd\zeta|
= \int^{\varphi_a }_{-\varphi_a } \exp( \cos \frac{\theta}{a}) d\theta 
\le \int^{\varphi_a }_{-\varphi_a } e d\theta = 2\varphi_a e \le 2\pi e
$$
for all $a\in[a_-,a_+]$, and $b\in[b_-,b_+]$.
Therefore, 
\begin{equation}\label{Cestimate}
\sup_{a\in[a_-,a_+],\,b\in[b_-,b_+],\,\, p=0,1}
C_{\gamma_1,\gamma_2,p} =: C_2 < \infty.           
\end{equation}
Thus, by virtue of the definition of the constant \(C_{\gamma_1,\gamma_2,p}\) and the estimate (\ref{Estimate_E}), we obtain the formulation of the lemma in Case 1 for all $t\geq 1$.

\textbf{Case 2:} \(a=b\). First we note
$$
\frac{1}{\zeta+t} = \frac{1}{t} - \frac{\zeta}{t^2}
+ \frac{\zeta^2}{t^2(\zeta+t)},                              $$
and
$$
\frac{1}{2\pi a i}\int_{\Gamma_{\varphi_a}}
  \exp( \zeta^{\frac{1}{a}})\zeta^{\frac{1}{a}} d\zeta = \frac{1}{\Gamma(-a)},
$$
  $$
 \frac{1}{2\pi a i}\int_{\Gamma_{\varphi_a}}
  \exp( \zeta^{\frac{1}{a}})\zeta^{\frac{1}{a}-1} d\zeta =0.
$$
Therefore, we may write
\begin{align*}
& E_{a,a}(-t) = \frac{1}{2\pi ai}\left(\int_{\Gamma_{\varphi_a}}
e^{\zeta^{1/a}}
 \zeta^{1/a -1}\dd\zeta\right) \frac{1}{t}
- \frac{1}{2\pi ai} \left(\int_{\Gamma_{\varphi_a}}
 e^{\zeta^{1/a}}
 \zeta^{1/a }\dd\zeta \right)\frac{1}{t^2}\\
+& \frac{1}{2\pi ai}\left( \int_{\Gamma_{\varphi_a}}
 \frac{e^{\zeta^{1/a}}
 \zeta^{1/a+1}}{\zeta+t}\dd\zeta \right)\frac{1}{t^2}.
\end{align*}
In view of (\ref{sin}), we obtain
$$
\left\vert E_{a,a}(-t) + \frac{1}{\Gamma(-a)}
\frac{1}{t^2}\right\vert 
\le \frac{1}{2\pi a t^3 c_0}
\int_{\Gamma_{\varphi_a}}\left\vert 
e^{\zeta^{1/a}}
 \zeta^{1/a+1}\dd\zeta
\right\vert.
$$
By (\ref{sin}) and (\ref{Cestimate}), we see 
$$
\sup_{a\in[a_-,a_+]} \left\vert E_{a,a}(-t) 
+ \frac{1}{\Gamma(-a)}\frac{1}{t^2}\right\vert
\le \frac{1}{2\pi a t^3 c_0} C_{a,a,2}
\le \frac{C_2}{2\pi c_0 a_-}\frac{1}{t^3}
=: \frac{C_3}{t^3}.
$$
Thus, we obtain the formulation of the lemma in Case 2 for all $t\geq 1$.

It remains to consider $0\le t\le1$.  By \eqref{eq:ML-definition},
\begin{equation*}
 \abs{E_{a,b}(-t)}
 \le\sum_{k=0}^\infty\frac{1}{\Gamma(ak+b)}.
\end{equation*}
Choose $N$ such that $a_-k+b_-\ge2$ for $k\ge N$.  Since $\Gamma$ is
increasing on $[2,\infty)$,
\begin{equation*}
 \frac{1}{\Gamma(ak+b)}
 \le\frac{1}{\Gamma(a_-k+b_-)},
 \qquad k\ge N.
\end{equation*}
The majorant series converges by Stirling's formula, while the first $N$ terms
are uniformly bounded on the compact parameter rectangle.  Hence
$\abs{E_{a,b}(-t)}\le C$ for $0\le t\le1$. 

 Thus estimate \eqref{eq:ML-uniform} is proved.
\end{proof}

\end{document}